\documentclass[12pt,req]{article}
\usepackage{fancyhdr}
\usepackage[mathscr]{eucal}
\usepackage{amsmath}
\usepackage{amssymb}
\usepackage{amsthm}
\usepackage{float}
\usepackage{caption}
\usepackage[caption = false]{subfig}
\usepackage{enumerate}
\usepackage{fullpage}
\usepackage{graphicx}
\usepackage{multicol}
\usepackage{tikz,  mathtools, soul, mathdots, yhmath}
\usetikzlibrary{calc,decorations.markings,arrows}

\usepackage{hyperref}

\theoremstyle{plain}
\newtheorem{theorem}{Theorem}
\newtheorem{corollary}[theorem]{Corollary}

\newtheorem{proposition}[theorem]{Proposition}
\newtheorem{identity}[theorem]{Identity}

\theoremstyle{definition}

\newtheorem{example}[theorem]{Example}

\theoremstyle{remark}
\newtheorem{remark}[theorem]{Remark}

\title{Arithmetics of some  sequences via $2$-determinants} 
\author{Dusko Bogdanic and Milan Janji\'c\footnote{Department of Mathematics and Computer Science, 
University of Banja Luka, Bosnia and Herzegovina 
dusko.bogdanic@pmf.unibl.org, milan.janjic@pmf.unibl.org}}
\date{} 

\begin{document}

\maketitle

\begin{abstract} 
We extend our investigation of $2$-determinants, which we
defined in a previous paper. For a linear homogenous recurrence of the second
order, we consider relations between different sequences satisfying the same linear homogeneous recurrence of
the second order. After we prove a generalized identity of d'Ocagne, we derive, from a single identity,  a number of classical identities (and their generalizations) such as d'Ocagne's, Cassini's, Catalan's, and Vajda's. Along the way, the corresponding combinatorial interpretations in terms of restricted words over a finite alphabet are stated for  the sequences we investigate. 
\end{abstract}

\maketitle

\section{Introduction and preliminaries}
We continue  our investigation of $2$-determinants,
defined in \cite{ja1}. For a linear homogenous recurrence of the second order, we consider
relations between different sequences satisfying the same recurrence. Linear homogenous recurrences of the second order are much studied objects and there is an overwhelming amount of literature containing various formulas involving sequences defined recurrently (as an introduction to the topic, we recommend \cite{Va}, \cite{JO}, \cite{BQ},  and \cite{Du}).

Our method produces a number of identities involving Fibonacci numbers and polynomials, bisection of Fibonacci numbers, positive integers,  Pell numbers, Jacobsthal numbers, Mersenne numbers, and Chebyshev polynomials of the second kind.
For these and some other objects, we derive variations   of well-known identities, such as d'Ocagne's, Cassini's, Vajda's, and Catalan's identities.

Let $x$ and $y$ be integer-valued variables.
We consider the following recurrence of the second order:
\begin{equation}\label{rr}a_{n+1}(x,y)=x\cdot a_{n}(x,y)+y\cdot a_{n-1}(x,y),\,\,\, n>0,\\
\end{equation}
where \begin{equation}\label{rr2} a_{0}(x,y)=0,\,\, a_1(x,y)=1. \end{equation}

We investigate mutual connections between three sequences $(a),(b),$ and $(c)$ which satisfy
(\ref{rr}), and where $(a)$ also satisfies (\ref{rr2}). Our main result is the following general identity. 
\begin{theorem}[A generalized identity of d'Ocagne] Let $(a)$ be a sequence satisfying Equation~(\ref{rr}) and Equation~(\ref{rr2}), and let
$(b)=(b_0,b_1,\ldots)$ and $(c)=(c_0,c_1,\ldots)$ both satisfy Equation~(\ref{rr}). Then
\begin{equation}\label{ocintr}
\begin{vmatrix}b_k&b_{k+m}\\c_k&c_{k+m}
\end{vmatrix}=(-y)^{k}\cdot a_{m}\cdot\begin{vmatrix}b_{0}&b_{1}\\c_{0}&c_{1}
\end{vmatrix}.
\end{equation}
\end{theorem}

Once we establish the identity (\ref{ocintr}), from this single identity we will derive a number of identities involving some classical sequences of numbers.  We  note that this formula is, in a sense, asymmetric, since  factors $(-y)^k$
and $a_m$ appear only on the right-hand side of the equation. Moreover, on
the right-hand side  only two initial terms of sequences $(b)$ and $(c)$ appear.

We also note that the left-hand side of (\ref{ocintr}) contains two arbitrary parameters $k$
and $m$.  Furthermore,  in Theorem~\ref{fpt}, a formula containing four
different parameters will be proved. Extra parameters allow us to derive a number of identities
concerning classical sequences as special cases of this particular identity.

We see that the fundamental role in our investigation is played by sequences that satisfy (\ref{rr}) and (\ref{rr2}). In this section we closely investigate such sequences. In particular, we give explicit
formulas for them and combinatorial interpretations of such sequences in
terms of restricted words over a finite alphabet.

In the cases under consideration, $x$ and $y$ will always have fixed values, so that
we can write $a_n$ instead of $a_n(x,y)$, omitting $x$ and $y$ to simplify notation.

The following result is proved in Proposition~11, Proposition~12, and Proposition~16 in \cite{ja2}.

\begin{proposition} Let $(a)$  be a sequence satisfying (\ref{rr}) and (\ref{rr2}). 
\begin{enumerate}
    \item The following explicit formula holds
\begin{equation}\label{pp1}
a_n=\sum_{k=0}^{\lfloor\frac{n-1}{2}\rfloor}{n-1-k\choose k}\cdot x^{n-2k-1}\cdot
y^k.
\end{equation}
\item
    If $x > 0$ and $y > 0$, then $a_{n+1}$ equals the number
of words of length $n$ over the alphabet
$\{0,1,\ldots,x-1,x,\ldots,x+y-1\}$ in which letters $0,1,\ldots,x-1$
avoid runs of odd lengths.
\item
If $x>0$, $y<0$, and $-y<  x$, then $a_n$ is the number of words of length $n-1$
over
$\{0,1,\ldots, x-1\}$
with no subword of the form $0i$, where $i\in \{1,2,\ldots, -y\}$.
\end{enumerate}
\end{proposition}

Some well-known integer sequences  are given by a linear homogeneous recurrence of the second order, for instance, Fibonacci numbers,  Fibonacci polynomials,  Jacobsthal numbers, and Pell numbers. They are
 obtained   when $x>0,\, y>0$. The first example concerning the case $x>0,\, y<0$ shows that positive
 integers also satisfy (\ref{rr}).  Chebyshev
 polynomials of the second kind, bisection of Fibonacci numbers, and Mersenne numbers
 also belong to this class.

In the next example, Equation~(\ref{pp1}) is used to give an explicit expression for each of the classical sequences involved. These formulas are well known and most of them have been widely described in numerous literature (\cite{Va}, \cite{Ko}, \cite{BQ}, \cite{qF}, \cite{qQ}, \cite{Ho}, \cite{Ho2}). Here, ten examples are listed  in order to emphasize the idea that all of them hold a common origin in Equation~(\ref{pp1}). Along the way, the corresponding combinatorial interpretations are stated for each of the sequences.

We start with the most important two: Fibonacci polynomials and Chebyshev polynomials of the second kind. Some of the examples that follow are just  particular cases  of these two.

\begin{example}\label{primjer}
\begin{enumerate}
\item If $x>0$ and $y=1$, then  $a_{n+1}=F_{n}(x), \, n\geq 1$,
where $F_{n}(x)$ is the $n$th Fibonacci polynomial. Also, Equation~(\ref{pp1}) gives the
 explicit expression for $F_{n}(x)$ (see \cite{qF}):
\begin{displaymath}
F_{n}(x)=\sum_{k=0}^{\lfloor\frac {n-1}{2}\rfloor}{n-k-1\choose k}x^{n-2k-1}.
\end{displaymath}
In terms of restricted words, if $x>0$ is an integer, then $F_{n}(x)$ equals the number of words
of length $n$   over $\{0,1,\ldots,x-1\}$ in which $0$ avoids runs of  odd lengths.
\item
If  $x=2z$ and $y=-1$, then $a_n=U_{n}(z)$, where $U_n(z)$ is the Chebyshev polynomial
of the second kind. From Equation~(\ref{pp1}), we get the following well-known formula (see \cite{qQ}):

\begin{displaymath}
U_{n}(z)=\sum_{k=0}^{\lfloor \frac {n-1}{2}\rfloor}(-1)^k\cdot{n-k-1\choose
k}(2z)^{n-2k-1}.
\end{displaymath}

In terms of restricted words, if $z>0$ is an integer, then
 $U_{n+1}(z)$ equals the number of  words of length $n$ over the alphabet
$\{0,1,\ldots,2z-1\}$ avoiding the subword $01$.

 \item Particular case of (1) is when  $x=1$ and $y=1$. In this case we have $a_{n+1}=F_{n},\, n\geq 1$.
Also, Equation~(\ref{pp1}) is the  standard  expression for the Fibonacci numbers in terms of
the binomial coefficients (see Identity (54) of \cite{Va}):
\begin{displaymath}
F_{n}=\sum_{k=0}^{\lfloor\frac {n-1}{2}\rfloor}{n-k-1\choose k}.
\end{displaymath}
Combinatorially, the Fibonacci number $F_{n}$ equals the number of binary words of
length $n$ avoiding a run of zeros of odd length.
\item For $x=2$ and $y=1$,  we have $a_{n+1}=P_{n}$, where $P_n, \, n\geq 0$, is the $n$th Pell
number. From
Equation~(\ref{pp1}), we get (see \cite{Ho})
\begin{displaymath}
P_{n}=\sum_{k=0}^{\lfloor\frac {n-1}{2}\rfloor}2^{n-2k-1}\cdot {n-k-1\choose k}.
\end{displaymath}
Also,  $P_{n}$ equals the number of ternary words of  length $n$ in which $0$
avoids runs of odd lengths. The Pell numbers are sometimes  called  ``silver Fibonacci numbers".
\item For $x=1$ and $y=2$, we have $a_{n+1}=J_{n}$, where $J_n,\, ( n=0,1,2,\ldots)$, are the
Jacobsthal numbers.
From Equation~(\ref{pp1}), we obtain  (see \cite{Ho2})
\begin{displaymath}
J_{n}=\sum_{k=0}^{\lfloor\frac {n-1}{2}\rfloor}2^k{n-k-1\choose k-1}.
\end{displaymath}

Also, the number $J_{n}$ equals the number of ternary words of length $n$ in which
$0$ and $1$ avoid runs of odd lengths.
\item If $x=2$ and $y=2$, then
$a_{n+1}$ is the number of ways to tile a board of length $n$ using tiles of two colors of length 1 and 2. Also,
\begin{displaymath}
a_{n+1}=\sum_{k=0}^{\lfloor\frac {n-1}{2}\rfloor}2^{n-k-1}\cdot {n-k-1\choose
k},\,\, n>0.
\end{displaymath}
\item If $x=2$ and $y=-1$, then  $a_0=0,\, a_1=1,$ and $a_{n+1}=2a_n-a_{n-1}$,
which is the recurrence for non-negative integers. Thus, we have 
\begin{displaymath}
n=\sum_{k=0}^{\lfloor\frac {n-1}{2}\rfloor}(-1)^k\cdot 2^{n-2k-1}{n-k-1\choose
k},\,\, n>0.
\end{displaymath}
This formula for $n$ may seem rather complex, but its combinatorial meaning is very
simple.
Namely, $n$ equals the number of binary words of length $n-1$ avoiding
$01$, which is obvious.
\item  If $x=3$ and $y=-1$, then we  have that $a_{n+1}$ is the bisection of Fibonacci numbers, that is, $a_{n+1}=F_{2n}$. From Equation~(\ref{pp1}), we obtain 
\begin{displaymath}
F_{2n}=\sum_{k=0}^{\lfloor\frac {n-1}{2}\rfloor}(-1)^k\cdot 3^{n-2k-1}\cdot
{n-k-1\choose
k}.
\end{displaymath}
Also, $F_{2n}$ equals the number of ternary words of length $n-1$ avoiding $01$.
\item If $x=3$ and $y=-2$, then $a_n=2^n-1$. These numbers are usually called 
Mersenne numbers.
We have 
\begin{displaymath}
2^{n}-1=\sum_{k=0}^{\lfloor\frac{ n-1}{2}\rfloor}(-2)^k\cdot 3^{n-2k-1}\cdot
{n-k-1\choose
k}.
\end{displaymath}
Also, $2^{n+1}-1$ equals the number of ternary words of length $n$ avoiding $01$ and $02$.

\item If $x=4$ and $y=-3$, then $a_n=\frac{3^n-1}{2}$.
Next we have
\begin{displaymath}
\frac{3^n-1}{2}=\sum_{k=0}^{\lfloor\frac{ n-1}{2}\rfloor}(-3)^k\cdot 4^{n-2k-1}\cdot
{n-k-1\choose
k}.
\end{displaymath}
Also, $\frac{3^n-1}{2}$ equals the number of
quaternary words of length $n$ avoiding $01, 02$ and $03$.

\end{enumerate}
\end{example}

\section{Identities}

Unless stated otherwise, throughout this section, we  assume that $(a)$ is a sequence satisfying Equation~(\ref{rr}) and Equation~(\ref{rr2}), and that
$(b)=(b_0,b_1,\ldots)$ and $(c)=(c_0,c_1,\ldots)$ both satisfy Equation~(\ref{rr}). Note that we do not have any assumptions about the initial conditions for the sequences $(b)$ and $(c)$.

In our previous paper \cite{ja1}, we defined the notion of an $n$-determinant, and used it to derive numerous identities related to some sequences given by linear homogeneous recurrences of the second order. Here, we use the results obtained via $2$-determinants to derive a number of identities related to some classical sequences. For the definition of an $n$-determinant and related results, we refer the reader to \cite{ja1}. We remark here that the matrix methods are widely used when it comes to proving some of the classical identities. For example, see \cite{JO} for a demonstration of how powerful these methods can be in simplifying proofs of some of the identities.

The starting point of our investigation is the following general theorem that relates sequences satisfying the same linear homogeneous recurrence of the second order.

\begin{theorem}[\cite{ja2}, Proposition 8] Let $(u)=(u_0,u_1,\ldots)$ and $(v)=(v_0,v_1,\ldots)$ be any two sequences, and let $(b)=(b_0,b_1,\ldots)$ and $(c)=(c_0,c_1,\ldots)$ be two sequences both satisfying  the same recurrence:
\begin{align*}b_{n+1}&=u_{n}b_{n}+v_{n-1}\cdot b_{n-1},\,\,\, n>0,\\
                     c_{n+1}&= u_{n}c_{n}+v_{n-1}\cdot c_{n-1},\,\,\, n>0.
\end{align*}
Then, for $k\leq n+2$, we have 
\begin{equation}\label{prop8}
\begin{vmatrix}b_k&b_{n+2}\\c_k&c_{n+2}
\end{vmatrix}=(-1)^{k}\cdot v_1\cdot v_2\cdots v_k\cdot  a_{n-k+2}\cdot\begin{vmatrix}b_{0}&b_{1}\\c_{0}&c_{1}
\end{vmatrix},
\end{equation}
where $a_0=0$, $a_1=1$, $a_2=u_{k+1}$, $a_i=v_{k+i-2}a_{i-2}+u_{k+i-1}a_{i-1}$, $i>2$.
\end{theorem}

The basic result that we use to investigate sequences of numbers is the following  direct corollary of the previous theorem.
\begin{theorem}[A generalized identity of d'Ocagne] Let $(a)$ be a sequence satisfying Equation~(\ref{rr}) and Equation~(\ref{rr2}), and let
$(b)=(b_0,b_1,\ldots)$ and $(c)=(c_0,c_1,\ldots)$ both satisfy Equation~(\ref{rr}). Then
\begin{equation}\label{oc}
\begin{vmatrix}b_k&b_{k+m}\\c_k&c_{k+m}
\end{vmatrix}=(-y)^{k}\cdot a_{m}\cdot\begin{vmatrix}b_{0}&b_{1}\\c_{0}&c_{1}
\end{vmatrix}.
\end{equation}
\end{theorem}
\begin{proof} The statement follows from the previous theorem by setting, for all $n$, $u_n=x$, $v_n=y$, and $n+2=k+m$.
\end{proof}

The importance of Equation~(\ref{oc}) lies in the fact that there is an extra term $a_m$ on the right-hand side of the formula. This gives us a lot of freedom in choosing concrete values of the sequence $a_m$, yielding many identities as special cases of  Equation~(\ref{oc}). We will derive some of these identities now. Note that  when it comes to the indices that appear on the left-hand side of Equation~(\ref{oc}), they are not as general as indices appearing in some well-known identities, such as Identity $(70)$ of \cite{JO}, but the extra term $a_m$ on the right-hand side makes up for the lack of full generality of the indices. Moreover, there are four different parameters appearing as indices in the identity given in Theorem~\ref{fpt}. Extra parameters allow us to derive a number of identities concerning classical sequences as special cases of this particular identity.

In Section 4 of \cite{ja1}, we stated a number of d'Ocagne's identities for Fibonacci
numbers and polynomials, Lucas and Chebyshev polynomials.
We illustrate Equation~(\ref{oc}) by
deriving identities for some sequences described in the preceding section.

To clarify the name of the identity given by Equation~(\ref{oc}), we prove that
d'Ocagne's identity for Fibonacci
numbers is a particular case of  this identity.
\begin{corollary}[d'Ocagne's identity] The following formula holds
\begin{displaymath}
\begin{vmatrix}F_{k+1}&F_{k+m+1}\\F_{k}&F_{k+m}\end{vmatrix}=(-1)^{k}\cdot F_{m},
\end{displaymath}
\end{corollary}
\begin{proof}
When $x=y=1$,  Equation (\ref{oc}) becomes
the recurrence for  Fibonacci numbers. Hence, $a_{m}=F_{m}$,   $m\geq 0$.
For sequences $(b)$ and  $(c)$, we again choose the Fibonacci numbers with the
initial
conditions such that the
determinant on the right-hand side of Equation~(\ref{oc}) is equal to $1$. For
instance, we
can choose $c_0=0,c_1=1$, and $b_0=1,b_1=1$, that is,
$b_k=F_{k+1}$  and $c_k=F_{k}$. We thus obtain
 d'Ocagne's identity.
\end{proof}

We note one more consequence of Equation~(\ref{oc}). Namely, if we know sequences $(b)$ and $(c)$ from the previous theorem, we can determine the sequence $(a)$. 
\begin{corollary}If $(a)$, $(b)$, and $(c)$ are sequences satisfying (\ref{oc}), then the
members of the sequence $(a)$ are  rational functions of numbers
$b_i,c_i,(i=0,1,\ldots)$, with denominator equal to $(-y)^k\cdot(b_0c_1-b_1c_0)$.
\end{corollary}
\begin{remark}
Besides d'Ocagne's identity, three of the most important identities
are: Cassini's, Catalan's, and Vajda's. All these identities can be derived from the generalized identity of d'Ocagne (\ref{oc}).
What we show here is that these important identities hold for each sequence satisfying (\ref{rr}) and (\ref{rr2}), i.e.\ they are, in a sense, consequences of the homogenous linear recurrence of the second order, not of the particular coefficients.
\end{remark}

For  $m=1$, we have $a_1=1$, so that from Equation~(\ref{oc}) we get the following identity (see Identity (70) of \cite{JO}).
\begin{proposition}[A Cassini-like identity] The following formula holds
\begin{displaymath}
\begin{vmatrix}b_k&b_{k+1}\\c_k&c_{k+1}\end{vmatrix}
=(-y)^{k}\cdot\begin{vmatrix}b_0&b_{1}\\c_0&c_{1}\end{vmatrix}.
\end{displaymath}
\end{proposition}

It is easy to see that for  $x=y=1,b_0=1,b_1=1$, $c_0=0,$ and $c_1=1$,
we obtain the  standard Cassini's identity for Fibonacci numbers.

We now consider the Lucas numbers $L_0=2,L_1=1,L_2=3,\ldots$. We know that these numbers satisfy the same recurrence as the Fibonacci numbers do.
We thus take $x=y=1$ and obtain the following identity (see Identity $(16b)$ of \cite{Va}).
\begin{identity}[A Cassini-like identity for Fibonacci and Lucas numbers]
    \begin{displaymath}
  L_k\cdot F_{k+1}-L_{k+1}\cdot F_k=2\cdot (-1)^k.
    \end{displaymath}
\end{identity}

   In Example \ref{primjer},  we have seen  that the Jacobsthal numbers $J_n,(n=0,1,2,\ldots)$, are obtained from Equation~(\ref{pp1}) for $x=1,y=2$.  If we take $b_0=J_1,b_1=J_2,c_0=J_0,c_1=J_1$, we obtain
    the following identity (see Identity $(2.5)$ of \cite{Ho2}).
    \begin{identity}[A Cassini-like identity for Jacobsthal numbers]
    \begin{displaymath}
  {J }^2_{k+1}-J_k\cdot  J_{k+2}=(-2)^{k}.\end{displaymath}
\end{identity}

Since Chebyshev polynomials $T_n(z)$ of the first kind satisfy the same recurrence as $U_n(x)$ do, by taking
\begin{displaymath}
T_0(0)=1,T_1(1)=z, x=2z,y=-1,
\end{displaymath}
we obtain the following identity. 
\begin{identity}[A Cassini-like identity for Chebyshev polynomials of the first kind]
    \begin{displaymath}
T^2_{k+1}(z)-T_k(z)\cdot T_{k+2}(z)
=1-z^2.\end{displaymath}
\end{identity}

 Next, we assume that $k\geq p$. If we replace $k$ by $k-p$ in Equation~(\ref{oc}), then we get
\begin{displaymath}
\begin{vmatrix}b_{k-p}&b_{k+m-p}\\c_{k-p}&c_{k+m-p}\end{vmatrix}=(-y)^{k-p}\cdot
a_{m}\cdot
\begin{vmatrix}b_0&b_{1}\\c_0&c_{1}\end{vmatrix}.
\end{displaymath}
If we apply Equation~(\ref{oc}) to the right-hand side of the previous equality, we obtain the following universal
property. This is a well known index reduction formula (cf.\ Identity $(70)$ of \cite{JO}).
 \begin{proposition}[Index reduction formula] If $k\geq p$, then
\begin{displaymath}
\begin{vmatrix}b_k&b_{k+m}\\c_k&c_{k+m}\end{vmatrix}=(-y)^{p}\cdot
\begin{vmatrix}b_{k-p}&b_{k-p+m}\\c_{k-p}&c_{k-p+m}\end{vmatrix}.
\end{displaymath}
\end{proposition}
In particular, if $k=p$, then by using the index reduction formula, we can  write d'Ocagne's identity in the form:
\begin{displaymath}
\begin{vmatrix}b_k&b_{k+m}\\c_k&c_{k+m}
\end{vmatrix}=(-y)^{k}\cdot \begin{vmatrix}b_{0}&b_{m}\\c_{0}&c_{m}\end{vmatrix}.
\end{displaymath}
By comparing the last equality with Equation~(\ref{oc}), we get the following identity.
\begin{proposition}[A reduced identity of d'Ocagne]
\begin{displaymath}
a_m\cdot\begin{vmatrix} b_{0}&b_{1}\\c_{0}&c_{1}\end{vmatrix}=
\begin{vmatrix} b_{0}&b_{m}\\c_{0}&c_{m}\end{vmatrix}.
\end{displaymath}
\end{proposition}

We illustrate this formula with two identities. The first identity concerns Fibonacci numbers. 
\begin{identity} For arbitrary non-negative  integers $m,\, p$, and $q$, where $m>p$, the following holds
\begin{displaymath}
F_m\cdot\begin{vmatrix} F_p&F_{p+1}\\F_{q}&F_{q+1}\end{vmatrix}=
\begin{vmatrix} F_{p}&F_{m+p}\\F_{q}&F_{m+q}\end{vmatrix}.
\end{displaymath}
\end{identity}
The next identity concerns Chebyshev polynomials $U_n(x)$ of the second kind. 
\begin{identity}  For arbitrary non-negative  integers $m,\, p$, and $q$, where $m>p$, the following holds
\begin{displaymath}
U_m(x)\cdot\begin{vmatrix} U_p(x)&U_{p+1}(x)\\U_{q}(x)&U_{q+1}(x)\end{vmatrix}=
\begin{vmatrix} U_{p}(x)&U_{m+p}(x)\\U_{q}(x)&U_{m+q}(x)\end{vmatrix}.
\end{displaymath}
\end{identity}

In the next result, we introduce two more parameters in the formula
(\ref{oc}). In this theorem we take $(c)=(a)$.
\begin{theorem}[Four parameter theorem]\label{fpt} For $m\geq k,p\geq q$, the
following formula holds
\begin{displaymath}
 \begin{vmatrix}b_{k+p}&b_{m+p}\\a_{k+q}&a_{m+q}
\end{vmatrix}
=(-y)^{k+q}\cdot a_{m-k}\cdot b_{p-q}.
\end{displaymath}
\end{theorem}
\begin{proof}
We only need to use the index reduction formula twice.
We have
\begin{displaymath}
 \begin{vmatrix}b_{k+p}&b_{m+p}\\a_{k+q}&a_{m+q}\end{vmatrix}=(-y)^k\cdot
\begin{vmatrix}b_{p}&b_{m+p-k}\\a_{q}&a_{m+q-k}
\end{vmatrix}.
\end{displaymath}
Applying the index reduction formula on the right-hand side of the last equation, we
obtain
\begin{displaymath}
 \begin{vmatrix}b_{k+p}&b_{m+p}\\a_{k+q}&a_{m+q}\end{vmatrix}=(-y)^{k+q}\cdot
\begin{vmatrix}b_{p-q}&b_{m+p-k-q}\\a_{0}&a_{m-k}
\end{vmatrix},
\end{displaymath}
and since $a_0=0$, the assertion follows.
\end{proof}

By replacing $m$ by $m+k$, $p$ by $p+q$, and finally $k+q$ by $k$, we obtain  the following  corollary of the previous  theorem.
\begin{identity}[{A Vajda-like identity}] 
For sequences $(a)$ and $(b)$ we have
\begin{equation}\label{tri}
\begin{vmatrix}b_{k+p}&b_{k+m+p}\\a_k&a_{k+m}
\end{vmatrix}=(-y)^{k}\cdot a_{m}\cdot b_p.
\end{equation}
\end{identity}
It is clear that, if $x=1,y=1$ and $a_i=F_i,\, b_i=F_i\, (i=0,1,\ldots)$, we obtain the well known Vajda's
identity for Fibonacci numbers.

Also, it is clear that, if $x=1,y=1$ and $a_i=F_i,\, b_i=L_i\, (i=0,1,\ldots)$,
where $L_i$ is the $i$th Lucas number, we obtain the following identity (see Identity (19b) in \cite{Va}).
\begin{identity}[A Vajda-like identity for Fibonacci and Lucas numbers]
\begin{displaymath}
\begin{vmatrix}L_{k+p}&L_{k+m+p}\\F_k&F_{k+m}
\end{vmatrix}=(-1)^{k}\cdot F_{m}\cdot L_p.
\end{displaymath}
\end{identity}

The following three identities are special cases of a Vajda-like identity for non-Fibonacci numbers.

\begin{identity}[A Vajda-like identity for Mersenne numbers] If $a_i=2^{i-1},(i\geq
1)$, then
\begin{displaymath}
\begin{vmatrix}2^{k+p}-1&2^{k+m+p}-1\\2^k-1&2^{k+m}-1
\end{vmatrix}=2^k\cdot(2^m-1)\cdot(2^p-1).
\end{displaymath}
\end{identity}

\begin{identity}[A Vajda-like identity for positive integers] If $a_i=i$, $i\geq 0,$ then
\begin{displaymath}
\begin{vmatrix}k+p&k+m+p\\k&k+m
\end{vmatrix}=m\cdot p.
\end{displaymath}
\end{identity}

 If we take $(a)=(b), m=p=r,k=n-r$ in (\ref{tri}), we obtain the following identity.

\begin{identity}[A Catalan-like identity]
\begin{displaymath}
\begin{vmatrix}a_{n}&a_{n+r}\\a_{n-r}&a_{n}
\end{vmatrix}=(-y)^{n-r}\cdot a_{r}^2.
\end{displaymath}
\end{identity}

It is clear that this identity generalizes the standard Catalan's identity for
Fibonacci numbers, which is obtained for $(a)=(F)$ and $y=1$.
We illustrate this case with several identities. The first identity is for Jacobsthal numbers (cf.\ \cite{Ho2}).
\begin{identity}[A Catalan-like  identity for Jacobsthal numbers] If $a_n=J_{n}$, then
\begin{displaymath}
J_{n}^2-J_{n-i}\cdot J_{n+i}=(-2)^{n-i}\cdot J_{i}^2.
\end{displaymath}
\end{identity}
\begin{identity}[A Catalan-like identity for Pell numbrs] If $a_n=P_{n}$, then
\begin{displaymath}
P_{n}^2-P_{n-i}\cdot P_{n+i}=(-1)^{n-i}\cdot P_{i}^2.
\end{displaymath}
\end{identity} See \cite{Ho} for this and some other identities involving Pell numbers.
\begin{identity}[A Catalan-like identity for $U_n(z)$] Assume that $a_n=U_{n}(z)$.
Then
\begin{displaymath}
U_{n}^2(z)-U_{n-i}(z)\cdot U_{n+i}(z)=(-1)^{n-i}\cdot U_{i}^2(z).
\end{displaymath}
\end{identity}
See \cite{Ud} for the direct proof of the previous identity.
\vspace{2mm}

Finally, we illustrate the identity from Theorem~\ref{fpt} by two well known identities for Fibonacci numbers. For the first identity, we set  $(a)=(b)=(F),\, p=m,\, q=k$ and $y=1$,  then we have the following result (\cite{Ru}, page 77).
 \begin{identity}
 \begin{displaymath}F_{k+m}^2-F_{m-k}^2=F_{2k}\cdot F_{2m}.
\end{displaymath}
\end{identity}
Also, if $p=m+1,q=k+1$, and $y=1$, then we have the identity from \cite{Sh}, page 63.
\begin{identity}
\begin{displaymath}F_{k+m+1}^2+F_{m-k}^2=F_{2k+1}\cdot F_{2m+1}.
\end{displaymath}
\end{identity}

\end{document}